\documentclass[12pt,oneside]{article}

\usepackage[margin=1.25in]{geometry}
\usepackage{setspace}
\usepackage[numbers,sort&compress]{natbib}

\usepackage{amsmath}
\usepackage{amsthm}
\usepackage{amsfonts}
\usepackage{amssymb}
\usepackage{mathtools}
\usepackage{bm}

\usepackage{graphicx}
\usepackage{subcaption}
\usepackage{float}
\usepackage{enumerate}
\usepackage{lmodern}
\usepackage{xcolor}
\usepackage{needspace}

\usepackage{tikz}
\usetikzlibrary{arrows.meta,positioning,matrix}

\usepackage[colorlinks=true,linkcolor=blue,citecolor=blue]{hyperref}

\numberwithin{equation}{section}

\newcommand{\R}{\mathbb{R}}
\newcommand{\E}{\mathbb{E}}

\newcommand{\calF}{\mathcal{F}}
\newcommand{\1}{\mathbf{1}}
\newcommand{\Prob}{\mathbb{P}}

\newcommand{\e}{\varepsilon}

\newcommand{\Cov}{\operatorname{Cov}}

\bmdefine\mub{\mu}
\bmdefine\etab{\eta}
\bmdefine\varthetab{\vartheta}
\bmdefine\betab{\beta}
\bmdefine\sigmab{\sigma}
\bmdefine\gammab{\gamma}
\bmdefine\Gammab{\Gamma}

 \definecolor{jan}{rgb}{0.0,0.3,0.8}

\theoremstyle{plain}
\newtheorem{thm}{Theorem}[section]
\newtheorem{lem}[thm]{Lemma}
\newtheorem{prop}[thm]{Proposition}
\newtheorem{cor}[thm]{Corollary}

\theoremstyle{definition}
\newtheorem{defn}{Definition}[section]

\theoremstyle{remark}
\newtheorem{rem}{Remark}[section]

\newcounter{boxnote}

\title{Forecasting and Manipulating the Forecasts of Others}

\author{
  Sam Babichenko\thanks{%
    Department of Statistics and Applied Probability, South Hall,
    University of California, Santa Barbara, CA 93106, USA.
    E-mail: \textit{sam@sbabichenko.com}.
    I am grateful to my advisor Tomoyuki Ichiba for guidance and
    support throughout this project. I thank Javier Birchenall for
    valuable discussions on the economic applications. Thiha Aung, Olivier Mulkin, Yan Lashchev, and Daniel Naylor provided encouragement and companionship
    throughout my doctoral studies.
    Replication code is available at
    \url{https://github.com/sbabichenko/Noise-State-Games}.}
}

\date{}

\begin{document}

\maketitle

\begin{abstract}
Finite-player dynamic games with dispersed private information are difficult
because actions both move payoffs and reshape what opponents learn, generating
hierarchies of beliefs about beliefs. This paper provides a recursive
representation for this problem. The noise state records agents' beliefs about
the underlying shocks that generate histories, so higher-order beliefs are
generated by composition rather than tracked as separate state variables. In the
canonical continuous-time LQG benchmark, the representation becomes explicit:
beliefs, value gradients, and policy rules are deterministic impulse-response
functions, and equilibrium is a deterministic fixed point in those functions.
Any fixed point in the noise-state linear class is a Nash equilibrium against
arbitrary admissible $L^2$ deviations. The first-order system contains an
information wedge, the shadow price of changing opponents' posteriors. In a
two-player benchmark, the wedge explains why pooling gains are mostly strategic,
why optimal precision allocation can starve an inefficient player of
information, and why signal precision changes policy rules themselves, so
separation fails.
\end{abstract}

\section{Introduction}

Dynamic games with dispersed private information pervade economics: firms infer competitors' pricing from market outcomes, traders learn from the order flow their own trades generate, central banks know their announcements reshape private agents' signals. In each setting, actions feed back into the information environment that agents optimize against, generating an infinite hierarchy of beliefs that Townsend~\cite{Townsend1983} identified and Sargent~\cite{Sargent1991} showed resists finite-dimensional analysis even in linear--Gaussian environments.

There is a gap between the way economic events are commonly understood and the class of dynamic information models that have been tractable. In ordinary economic reasoning, actions are read as signals: firms, banks, funds, workers, and policymakers interpret what others do, adjust, and thereby create new signals for others to interpret. Formal models often simplify this feedback by fixing the signal process, making individual actions negligible, using public histories, or truncating belief hierarchies. This paper closes the gap for linear--Gaussian environments: it keeps action--belief--signal feedback while replacing the hierarchy with a recursive state based on primitive shocks.

The key is to represent beliefs as beliefs about the underlying sources of randomness rather than as hierarchies of beliefs about the endogenous state. A researcher who specifies payoffs, a state equation, and a signal structure gets back equilibrium impulse-response functions; the belief hierarchy is handled internally.

We call this representation the noise state. The construction extends beyond the LQG setting studied here, though its explicit kernel form relies on conditional Gaussian filtering. More generally, it supplies a recursive state for dynamic games with dispersed information.

In a game of poker, the common approach describes each player's beliefs about opponents' hands, their beliefs about opponents' beliefs, and so on. The current approach instead characterizes each player's belief about the underlying source of randomness (here, the order the cards were shuffled) and how that belief depends on the truth; higher-order beliefs follow by composing these dependencies. The two representations are equivalent, but approximating a composition of deterministic dependencies is natural, whereas truncating a belief hierarchy is not.

A deviation affects the physical state, as in a perfect-information game, but it also changes what opponents infer about the primitive shocks. Those revised noise states feed into opponents' actions and then back into the physical system. Thus the relevant Markov state consists of the physical state, the players' noise states, and, when relevant, the filtering maps that translate histories into beliefs.

The LQG benchmark is the clean case in which this representation becomes explicitly solvable. Beliefs, value gradients, and policy rules reduce to deterministic impulse-response functions, and equilibrium is a fixed point in those functions. This benchmark is the base case for a broader noise-state program in conditional-Gaussian linear economies, including strategic trading, CARA risk aversion, and symmetric sparse networks in which local feedback modes coexist with aggregate mean-field modes.

The characterization produces a shadow price of manipulation, the information wedge $\mathcal{V}^i_t$. In the two-player benchmark, the wedge explains why most gains from pooling information are strategic rather than purely statistical, why optimal precision allocation can generate information starvation by concentrating signals on the efficient mover of the state, and why separation fails: changing signal precision changes equilibrium policy rules, not only posterior variances. As a diagnostic, Corollary~\ref{cor:exogenous_signal} shows that the wedge disappears when opponents' signals are exogenous to individual controls, recovering the standard fixed-information LQG optimality system. Section~\ref{sec:two_player_tracking_numerics} develops these economic consequences in a two-player game, and Section~\ref{sec:general_theory} develops the general theory.

\subsection{Related Literature}

This paper contributes to the rational-expectations literature on higher-order beliefs and dispersed information. Townsend's ``forecasting the forecasts of others'' problem and Sargent's signal-extraction examples show that private information can destroy finite-dimensional recursion even in linear environments \citep{Townsend1983,Sargent1991}. Existing approaches recover tractability by making signals exogenous to the best-response problem, making individual actions negligible, using public histories, or truncating the hierarchy. The present paper keeps the finite-player channel through which actions change what opponents learn, and represents it through beliefs about primitive shocks.

The paper also relates to the social value of information and imperfect common knowledge. In Morris--Shin, Angeletos--Pavan, Woodford, and related work, signal structures affect behavior but are fixed when agents optimize \citep{MorrisShin2002,AngelotosPavan2007,Woodford2003,BergemannHeumannMorris2015}. Here, information changes equilibrium policy kernels themselves, because the wedge prices shifts in another player's posterior; pooling gains are therefore largely strategic rather than purely statistical.

Frequency-domain and finite-state methods exploit signal laws fixed in the best-response problem \citep{SargentHansen1981,Kasa2000,RondinaWalker2021,HuoTakayamaWorking}. Here, a deviation changes the state process opponents observe, so filters are part of the equilibrium response. In the LQG benchmark, the noise state turns this joint filtering/control problem into a deterministic fixed point in impulse-response maps. Finite-order truncations approximate the hierarchy directly; noise states generate higher-order beliefs by composing primitive-shock projections \citep{Nimark2017}.

Technically, the paper connects to decentralized stochastic control, LQG team theory, and common-information methods. Classical team results obtain linearity under fixed information structures, while common-information methods condition on a shared information base \citep{Radner1962,MarschakRadner1972,NayyarMahajanTeneketzis}. The noise-state approach instead isolates the bilateral channel by which one player's action shifts another player's posterior, action, and subsequent signal; the wedge prices this channel directly.

Mean-field LQG games obtain tractability by making individual actions informationally negligible \citep{HuangMalhameCaines,CarmonaDelarue2018}. This is appropriate for aggregate modes, but it removes local informational channels; the wedge is the finite-player object that disappears. Symmetric sparse networks provide a different reduction: graph symmetries re-index the same noise-state representation into local feedback modes and aggregate mean-field interaction modes.

The precision-allocation exercise relates to information design and rational inattention \citep{KamenicaGentzkow,BergemannMorrisBCE,Sims2003}, but it does not solve for optimal attention or disclosure. It provides the equilibrium map from a fixed signal technology to behavior and welfare when precision changes policy rules through the wedge.

\section{Bilateral belief manipulation in a two-player game}
\label{sec:two_player_tracking_numerics}

This section uses a two-player tracking game as a microscope for the local feedback channel: one player's action changes the other's posterior, the response feeds back into the state, and the first player must price it. The example previews the objects developed in Sections~\ref{sec:model}--\ref{sec:general_theory} and shows why the wedge matters for welfare, pooling, precision allocation, and separation.

\subsection{Primitives and objectives}
\label{sec:ex_primitives}

Fix a horizon $[0,T]$. The state $X_t\in\mathbb{R}$ satisfies
\begin{equation}
\label{eq:ex_state_num}
dX_t = \big(D_t^1 + D_t^2\big)\,dt + \sigma\,dW_t^0,\qquad X_0=x_0,
\end{equation}
where $W=(W^0,W^1,W^2)$ is a standard three-dimensional Brownian motion and $\sigma>0$ is the fundamental volatility.
Player $i\in\{1,2\}$ does not observe $X$ directly; instead they
observe the private diffusion channel
\begin{equation}
\label{eq:ex_obs_num}
dY_t^1 = \sqrt{p_1}\,X_t\,dt + dW_t^1,\qquad
dY_t^2 = \sqrt{p_2}\,X_t\,dt + dW_t^2.
\end{equation}
The numbers $p_i$ are signal precisions.
Player~$i$ chooses $D^i$ adapted to their private observation
history $\mathcal{F}_t^i=\sigma(Y_s^i:s\le t)$.

Running costs are symmetric tracking losses with quadratic effort
penalty $r>0$:
\begin{equation}
\label{eq:ex_costs_num}
J^1(D^1;D^2)=\mathbb{E}\!\int_0^T\Big((X_t-1)^2 + r\,(D_t^1)^2\Big)\,dt,
\quad
J^2(D^2;D^1)=\mathbb{E}\!\int_0^T\Big((X_t+1)^2 + r\,(D_t^2)^2\Big)\,dt.
\end{equation}
Player~1 wants the state near~$+1$; player~2 wants it near~$-1$.
There is no terminal cost in this worked example.
\begin{figure}[t]
\centering
\begin{tikzpicture}[>=Stealth, thick, scale=0.85, transform shape,
  box/.style={draw, rounded corners, align=center, inner sep=5pt}]

\node[box] (W) at (0,0) {Primitive shocks\\ $W^0$ (fundamental), $W^1,W^2$ (signal noise)};

\node[box] (Y1) at (-5.5,-1.8) {Private signal $Y_t^1$\\ (precision $p_1$)};
\node[box] (X)  at (0,-1.8) {State / fundamental\\ $X_t$};
\node[box] (Y2) at (5.5,-1.8) {Private signal $Y_t^2$\\ (precision $p_2$)};

\node[box] (F1) at (-5.5,-3.6) {Filter / belief\\ $\widehat W_t^1(\cdot)$};
\node[box] (D)  at (0,-3.6) {Actions\\ $D_t^i=\bar D^i_t+\int_0^t D^i_t(u)\,d_u\widehat W_t^i(u)$};
\node[box] (F2) at (5.5,-3.6) {Filter / belief\\ $\widehat W_t^2(\cdot)$};

\draw[->] (W) -- (X);
\draw[->] (W) -- (Y1);
\draw[->] (W) -- (Y2);
\draw[->] (X) -- (Y1);
\draw[->] (X) -- (Y2);
\draw[->] (Y1) -- (F1);
\draw[->] (Y2) -- (F2);
\draw[->] (F1) -- (D);
\draw[->] (F2) -- (D);
\draw[->, bend right=30] (D.north) to (X.south);

\end{tikzpicture}
\caption{Information-to-action feedback in the worked example.}
\label{fig:ex_schematic}
\end{figure}

\subsection{Impulse responses and wedges}
\label{sec:ex_preview}

Equilibrium controls can be written in two equivalent impulse-response coordinates:
\begin{equation}
\label{eq:ex_control_representations}
\begin{aligned}
D_t^i
&= \bar D^i_t + \int_0^t D^i_t(u)\,d_u\widehat W_t^i(u),
\qquad
\widehat W_t^i(u):=\E[W_u\mid\calF_t^i],\\
&= \bar D^i_t + \int_0^t \mathcal D^i_t(u)\,dW_u .
\end{aligned}
\end{equation}
The first line is the strategy written in player~$i$'s noise-state coordinates.  The second line is the realized response to primitive shocks.  The deterministic functions $D^i$ and $\mathcal D^i$ are the objects solved by the fixed point, and fixed points are Nash equilibria against arbitrary admissible $L^2$ deviations.

\subsection{Computation}
\label{sec:ex_computation}

Equilibrium is computed by Picard iteration on the deterministic impulse-response system, discretized on a uniform grid of $N=40$ points; no Monte Carlo simulation is used.  Unless otherwise stated, $T=1$, $x_0=0$, $r=0.1$, $\sigma=1$, and there is no terminal cost.  The impulse-response benchmark in Figure~\ref{fig:impulse_w0_first} uses asymmetric precisions $(p_1,p_2)=(3,10)$.  The benchmark residual reaches $10^{-5}$ in $19$ iterations.\footnote{An interactive browser visualization is available at \url{https://sbabichenko.com/lqg}.}

Figure~\ref{fig:impulse_w0_first} displays the primitive object solved by the algorithm: a system-noise shock affects the state, players' posterior estimates catch up through private signals, and primitive-control responses feed the shock back into the state.

\begin{figure}[!htbp]
  \centering
  \includegraphics[width=\linewidth]{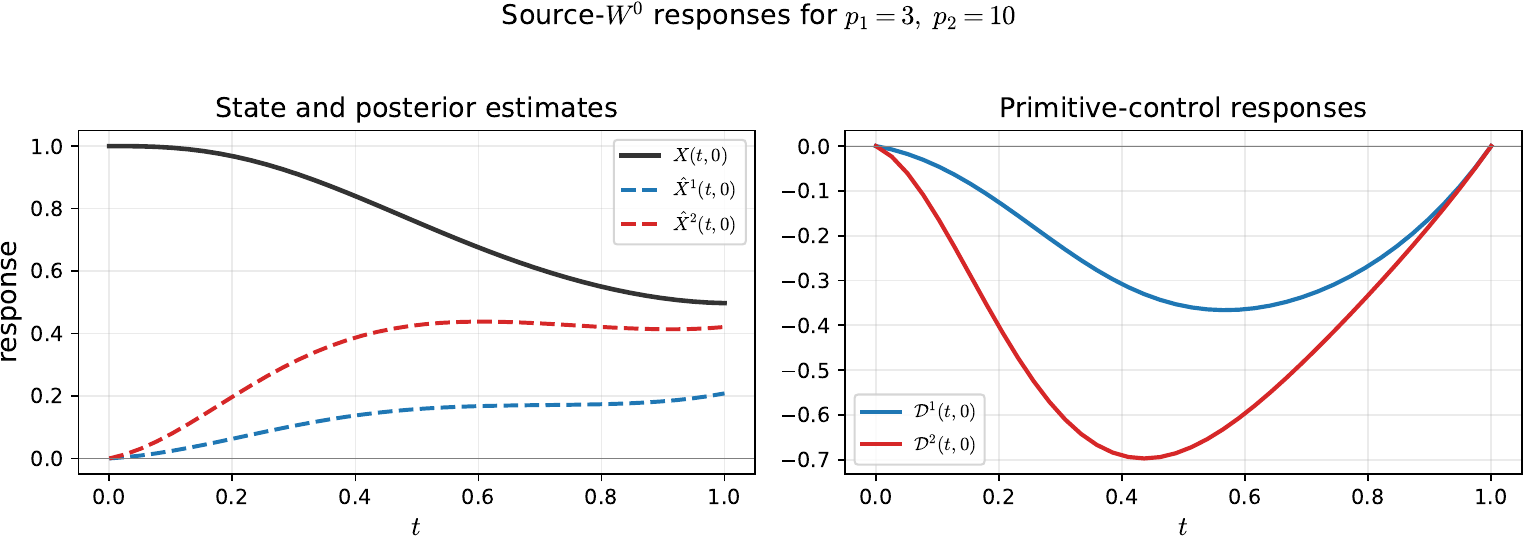}
  \caption{Impulse responses to the first system-noise shock $dW^0_0$ with $p_1=3$, $p_2=10$, and $r_1=r_2=0.1$.  Left: state and posterior-estimate responses.  Right: primitive-control responses $\mathcal D^1(t,0)$ and $\mathcal D^2(t,0)$.}
  \label{fig:impulse_w0_first}
\end{figure}

\subsection{Mean actions, separation failure, and the welfare
  cost of belief manipulation}
\label{sec:mean_overshoot}

Since the stochastic integral in the first line of~\eqref{eq:ex_control_representations} has zero
unconditional mean, $\mathbb E[D_t^i]=\bar D^i_t$ isolates the
deterministic tug-of-war over the target~$\pm 1$, while the response
term captures impulse responses to posterior shocks.

\paragraph{Separation failure.}
Under perfect information, the mean policy is independent of signal precision.
Decentralized information breaks this: precision changes the speed of opponents' updating, creating an incentive to act on posteriors rather than on the state.
The mechanism is the information wedge: changing signal precision changes the unresolved component, which changes the belief adjoint and hence the policy rule independently of the state estimate (Remark~\ref{rem:separation_failure}).

\begin{figure}[!htbp]
  \centering
  \includegraphics[width=0.66\linewidth]{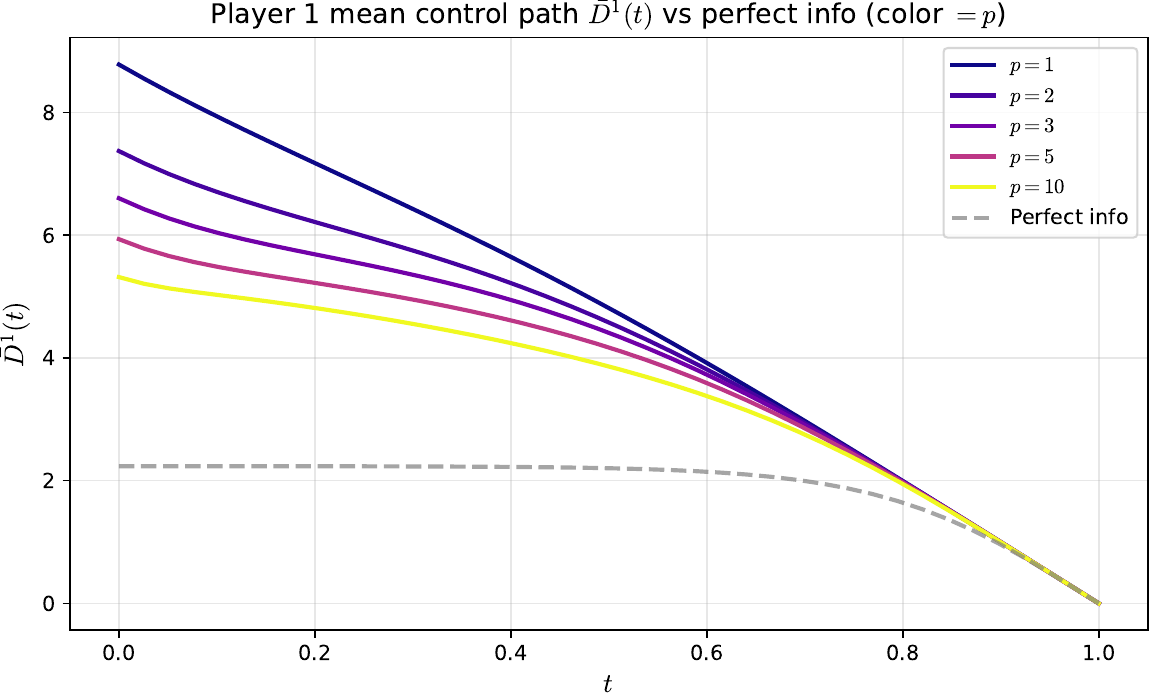}
  \caption{Equilibrium mean control $\bar D^1_t$ as signal precision $p$ varies, with the perfect-information benchmark dashed.}
  \label{fig:mean_response_precision}
\end{figure}

\paragraph{Information pooling as a Pigouvian intervention.}
Figure~\ref{fig:costs_private_vs_pooled} compares private signals with a pooled signal of precision $p_1+p_2$ under opposing targets $(\pm1)$ and a common target $(0)$.  Pooling is Pareto-improving in both cases, but the gain is an order of magnitude larger under competition.  Since the target change affects only the deterministic mean system, the cooperative case isolates the pure estimation benefit.  The excess gain under competition is the welfare cost of bilateral belief manipulation.

\begin{figure}[!htbp]
  \centering
  \includegraphics[width=\linewidth]{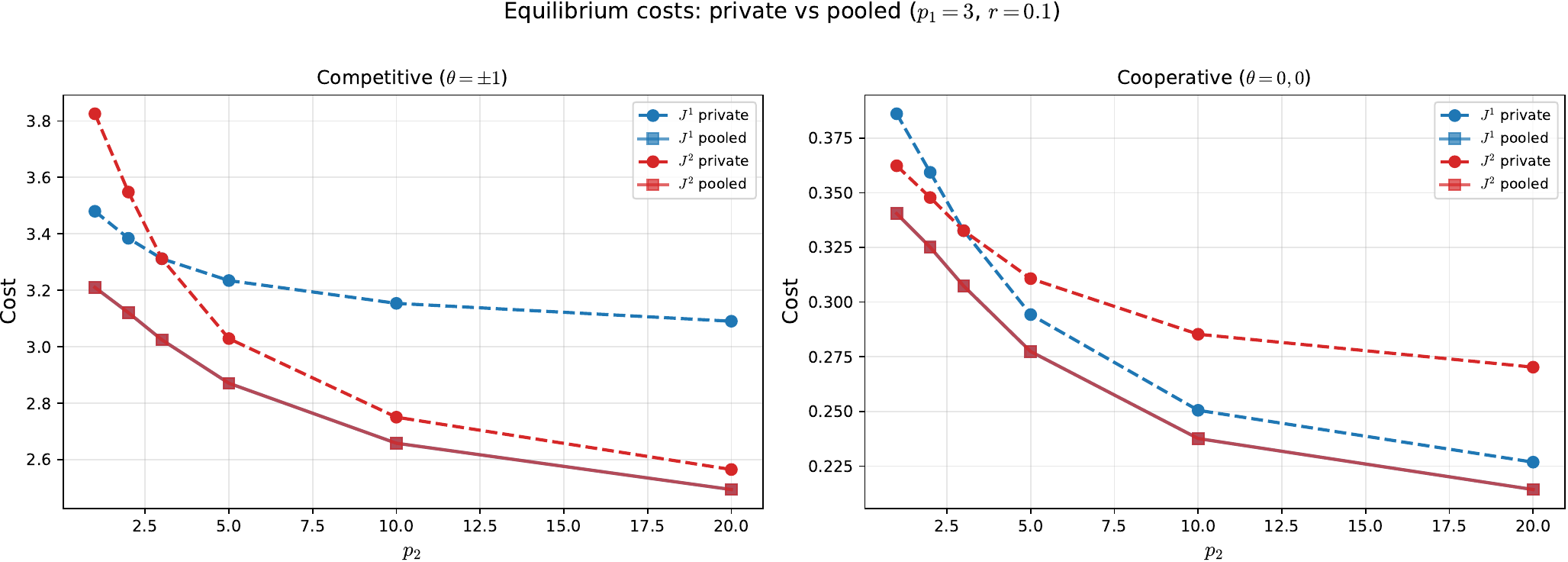}
  \caption{Equilibrium costs under private signals (dashed) and pooled signals (solid), with $p_1=3$ fixed and $p_2$ varying.  Left: opposing targets $(\pm1)$.  Right: common target $(0)$.}
  \label{fig:costs_private_vs_pooled}
\end{figure}

\subsection{Optimal precision allocation}
\label{sec:precision_allocation}

Information pooling is blunt. Given a fixed precision budget, how should a planner distribute it between the two players?

\Needspace{8\baselineskip}
\paragraph{Setup.}
Retain the state dynamics~\eqref{eq:ex_state_num} and observation
channels~\eqref{eq:ex_obs_num} with opposing targets
($\pm 1$), but allow asymmetric precisions
$(p_1,p_2)$ subject to $p_1+p_2=\bar P$ and asymmetric
effort costs $(r_1,r_2)$ with $r_1<r_2$, so that player~1 is
the more efficient mover of the state.
A planner chooses the split to minimize aggregate equilibrium
cost; for each candidate allocation we solve the full
equilibrium via Picard iteration.
Parameters: $\bar P=20$, $T=1$, with the fundamental volatility in~\eqref{eq:ex_state_num} reduced to $\sigma=0.5$.

\begin{figure}[H]
  \centering
  \includegraphics[width=\linewidth]{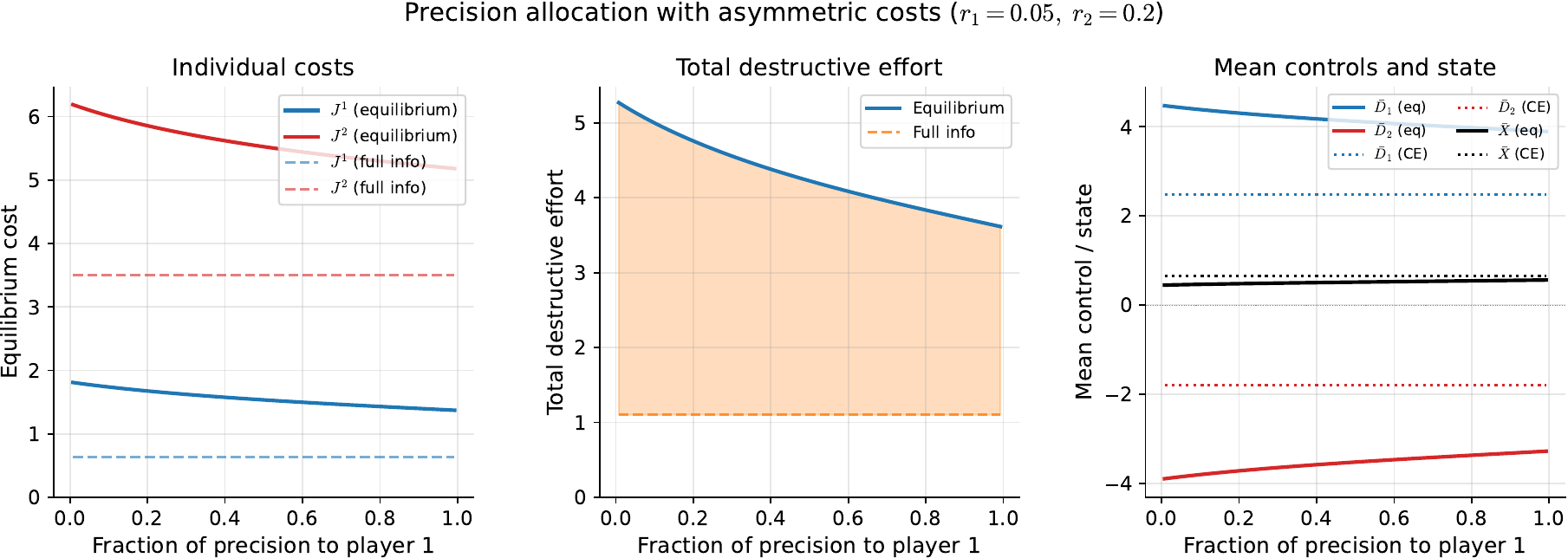}
  \caption{Equilibrium quantities across the precision split, $p_1+p_2=\bar P=20$, with opposing targets, $\sigma=0.5$, and asymmetric effort costs $(r_1,r_2)=(0.05,0.2)$.  Left: individual costs.  Center: total destructive effort.  Right: mean controls and mean state.}
  \label{fig:precision_decomposition}
\end{figure}

\paragraph{Information starvation.}
Under cooperation, precision should be assigned to the efficient mover.  Under competition, one might instead expect a planner to balance precision to limit manipulation.  Figure~\ref{fig:precision_decomposition} shows the opposite.  With asymmetric effort costs $(r_1,r_2)=(0.05,0.2)$, concentrating precision on the efficient player shifts the mean state toward that player's target, but it also sharply reduces destructive effort.  The latter effect dominates: both players' costs fall, and aggregate effort moves toward the full-information benchmark.

The mechanism is information starvation.  Denying the inefficient player signal access weakens the bilateral arms race, a commitment logic in the sense of \citet{Schelling1960}. Giving all precision to the inefficient player produces the worst outcome, because the cheap mover pushes hard against a well-informed opponent. The planner's optimum is therefore not balanced precision but an information monopoly for the efficient player.  The right panel also displays separation failure: certainty-equivalent controls are invariant to the split, whereas equilibrium controls move with the belief adjoints.

\section{The Decentralized LQG Game}
\label{sec:model}
\paragraph{Informal description of the game.}
The games we study have a common evolving state, private noisy observation channels, and controls that feed back into the state and hence into others' signals. These features generate the Townsend hierarchy; quadratic objectives and linear--Gaussian primitives make the problem tractable.

\subsection{State dynamics and objectives}
\label{sec:state_objectives}

Consider $n$ players interacting over a finite time horizon $[0,T]$ on a 
filtered probability space $(\Omega,\calF,\mathbb{F},\mathbb{P})$.

\paragraph{State dynamics.}
The state $X_t\in\R^d$ evolves as
\begin{equation}
\label{eq:state_dynamics}
dX_t=\bigl(A(t)X_t+\sum_{i=1}^n D_t^i\bigr)dt+\Sigma(t)\,dW_t^0,
\qquad X_0=x_0,
\end{equation}
where $A(t)\in\R^{d\times d}$ is deterministic and bounded,
$W^0$ is a $d$-dimensional Brownian motion,
$\Sigma(t)\in\R^{d\times d}$ is deterministic and bounded,
and $D_t^i\in\R^d$ is player $i$'s control.

\paragraph{Objectives.}
Let $G^{X,i}_t$ denote the linear state-cost coefficient. In the baseline case it is deterministic and bounded.\footnote{The formulas below also allow primitive-noise linear coefficients $G^{X,i}_t=\bar G^{X,i}_t+\int_0^t G^{X,i}_t(r)\,dW_r$ with deterministic mean and kernel; deterministic coefficients set $G^{X,i}_t(r)\equiv0$. The examples use deterministic coefficients.}
Each player $i$ incurs the cost
\begin{equation}
\begin{split}
\label{eq:random_cost}
\mathcal{C}^i(D^i; D^{-i}) :=& \int_0^T \left(X_t^\top G^{XX,i}(t) X_t 
+ 2(G^{X,i}_t)^\top X_t+G^i(t) +  (D^i_t)^\top G^{DD,i}(t) D^i_t\right) dt\\
&+ X_T^\top G^{XX,i}(T) X_T + 2(G^{X,i}_T)^\top X_T,
\end{split}
\end{equation}
where $G^{XX,i}(t)$ and $G^{DD,i}(t)$ are symmetric positive 
semidefinite matrices (with $G^{DD,i}(t)$ positive definite), and $D^{-i} := (D^j)_{j \neq i}$ denotes 
opponents' controls.
Each player $i$ seeks to minimize
\begin{equation}
\label{eq:objective}
J^i(D^i; D^{-i}) := \mathbb{E}\bigl[\mathcal{C}^i(D^i; D^{-i})\bigr].
\end{equation}
Opponents' controls enter player~$i$'s cost only through~$X_t$, which they move
via~\eqref{eq:state_dynamics} and which player~$i$ observes through~\eqref{eq:observations}.


\subsection{Information structure}
\label{sec:info_structure}

\paragraph{Observations.}
Player~$i$ does not observe $X_t$ directly.  Instead, their private information is summarized by the diffusion signal

\begin{equation}
\label{eq:observations}
dY^i_t = \Gamma^i(t)\, X_t \, dt + E^i\,dW_t,
\quad Y^i_0 = 0,
\end{equation}
where $\Gamma^i(t)\in\R^{d\times d}$ is a deterministic
observation gain and $E^i\in\R^{d\times(n+1)d}$ is the block
selector extracting player~$i$'s noise channel
($W^i_t:=E^iW_t$).
All Brownian motions $\{W^0,W^1,\ldots,W^n\}$ are mutually
independent.
The precision matrix $P^i(t):=\Gamma^i(t)^\top\Gamma^i(t)$
measures the Fisher information rate about the state per unit
time; the gain $\Gamma^i$ is the primitive, and $P^i$ is derived.
The block-selector form is without loss: replacing $E^i$ by a general invertible noise loading $\Sigma^i(t)$ amounts to substituting $\Gamma^{i,\top}(\Sigma^i\Sigma^{i,\top})^{-1}$ for $\Gamma^{i,\top}$ throughout.

\paragraph{Notation for primitive shocks.}
We collect all primitive noises into a single vector
$W_t:=(W^0_t,\dots,W^n_t)^\top\in\R^{(n+1)d}$.
The block projector $\Pi^i:=E^{i,\top}E^i$ is the identity on
the $i$-th block and zero elsewhere:
$\Pi^i v=(0,\dots,0,v^i,0,\dots,0)$.
The block-selector structure ensures that each player's
observation noise is an independent component of~$W$, giving the
noise-state its direct/indirect decomposition
(Theorem~\ref{thm:noise_state_filter_compact}): player~$i$
(almost) directly observes $\Pi^i W$ and infers $(I-\Pi^i)W$ from
drift-based learning.

Player $i$'s information at time $t$ is their observation history,
$\calF^i_t := \sigma(Y^i_s : s \leq t)$.
An admissible control is $\calF^i_t$-adapted with
$\mathbb{E}[\int_0^T \|D^i_t\|^2\, dt] < \infty$,
where $\|\cdot\|$ denotes the Euclidean norm on~$\R^d$.

\paragraph{Unresolved impulse response.}
For each player $i$, define
\[
\widetilde X^i_t(u) := X_t(u) - \widehat X^i_t(u),
\]
the component of the state's impulse response to the $u$-th shock that player $i$
has not yet resolved at time $t$: large when the shock is recent or the
observation channel weak, and shrinking as observations accumulate.
Formally, $\widetilde X^i_t(u)=\Cov(X_t,\,dW_u\mid\calF_t^i)/du$;
it serves as the filtering gain in the noise-state update
(Theorem~\ref{thm:noise_state_filter_compact}), with its explicit
formula in~\eqref{eq:Xtilde_from_F}.

Throughout the main characterization, the gain $\Gamma^i(\cdot)$ is fixed; the two-player benchmark studies comparative statics and planner-chosen precision within this fixed observation model.

\subsection{Nash Equilibrium}

Because each player's drift depends on opponents' estimates $\widehat X_t^k:=\E[X_t\mid\calF_t^k]$,
optimal actions generate an infinite hierarchy of conditional expectations.  The following equilibrium concept fixes opponents' strategy maps rather than their
realized actions, which is essential for the closure results in Section~\ref{sec:general_theory}.

\begin{defn}[Nash equilibrium in private-signal strategies]
\label{def:nash_strategy}
Throughout, a player's choice is a strategy map:
for each $t\in[0,T]$, the action $D_t^i$ is a (progressively measurable) functional of the
private observation history $Y^i_{\cdot\wedge t}$.
We write $D_t^i=D_t^i[Y^i]$, where the bracket notation $D_t^i[\cdot]$ denotes the map
from the observation path $(Y^i_s)_{s\le t}$ to the action at time~$t$.
We require $\E\int_0^T\|D_t^i\|^2\,dt<\infty$.

A profile $D^*=(D^{*,1},\ldots,D^{*,n})$ is a Nash equilibrium if for every player $i$ and every
admissible alternative strategy map $\widetilde D^i[\cdot]$,
\[
J^i(D^{*,i};D^{*,-i}) \le J^i(\widetilde D^i;D^{*,-i}),
\]
where the inequality compares the costs induced by the corresponding strategy profiles.
Equivalently: in a unilateral deviation by $i$, opponents' maps are held fixed and are
applied to the deviated signal paths.
\end{defn}

Holding opponents' strategy maps fixed is substantive because signals are endogenous: a unilateral deviation changes the drifts of opponents' observation paths, and hence the likelihoods they assign to histories.  Under admissible square-integrable controls this only changes drifts, so the induced observation laws are mutually absolutely continuous on every finite horizon.  A perfect-Bayesian formulation is therefore simpler than in discrete signal games: Bayes' rule applies on a common support, and the fixed strategy maps are evaluated on whatever observation paths occur.

\paragraph{Noise-state.}
On a given profile, the \emph{noise-state} is the conditional mean path of the aggregated primitive noise:
\begin{equation}
\label{eq:noise_state_def}
\widehat W_t^i(u):=\E[W_u\mid\calF_t^i],\qquad 0\le u\le t.
\end{equation}
Off equilibrium the same symbol denotes the fixed linear-Gaussian filter associated with the candidate profile, applied to the realized history $Y^i_{[0,t]}$.  On the profile it equals the conditional expectation in~\eqref{eq:noise_state_def}.  In a unilateral deviation, the filter is not recomputed.  At equilibrium, $\widehat W^i$ is a deterministic functional of player~$i$'s innovation path; since player~$i$'s own action enters $Y^i$ only through a known drift, that innovation path is invariant to $i$'s deviations.

\begin{defn}[(Primitive-noise) Linear Process]\label{def:linear_process}
An $L^2$ stochastic process $L=(L_t)_{t\in[0,T]}$ is a (primitive-noise) Linear process if there exist deterministic functions
\[
\bar L:[0,T]\to\R^m,
\qquad
L:[0,T]^2\to\R^{m\times (n+1)d},
\]
such that, for every $t\in[0,T]$,
\[
L_t \;=\; \bar L(t) + \int_0^t L(t,s)\, dW_s,
\qquad
\int_0^t \|L(t,s)\|^2\,ds < \infty .
\]
When convenient, we extend $L(t,s)$ by $0$ for $s>t$.
\end{defn}

\begin{defn}[Noise-state Linear control / strategy]\label{def:noise_state_linear_control}
An admissible strategy map is noise-state linear if, for the fixed filter described above, there exist deterministic functions
\[
\bar D^i:[0,T]\to\R^{d},
\qquad
D^i:[0,T]^2\to\R^{d\times (n+1)d},
\]
such that, for every private history $Y^i$ and every $t\in[0,T]$,
\[
D_t^i[Y^i]
=
\bar D^i_t
+
\int_0^t D^i_t(u)\,d_u\widehat W_t^i[Y^i](u),
\qquad
\int_0^t \|D^i_t(u)\|^2\,du<\infty.
\]
We extend $D^i_t(u)$ by $0$ for $u>t$ and omit the argument $[Y^i]$ on the equilibrium path.
\end{defn}


\section{General Theory}
\label{sec:general_theory}

This section derives the objects previewed in Section~\ref{sec:two_player_tracking_numerics} for the general $n$-player model.  The primitive-shock space plays the role of a dynamic Harsanyi type space \citep{Harsanyi1967}: conditional beliefs are projections of $W=(W^0,\ldots,W^n)$ onto private signal histories, and higher-order beliefs are obtained by composing deterministic response maps.  Because actions reshape the signals from which beliefs are formed, those maps must be solved as a fixed point.

\begin{figure}[H]
\centering
\begin{tikzpicture}[
  font=\small,
  >=Latex,
  box/.style={
    draw=black!55, line width=0.6pt,
    rounded corners=3pt,
    fill=black!2,
    align=center,
    text width=2.55cm,
    minimum height=0.8cm,
    inner sep=4pt
  },
  key/.style={box, draw=jan!70!black, fill=jan!8},
  arr/.style={->, line width=0.8pt, draw=black!70},
  soft/.style={->, line width=0.55pt, draw=black!45, dashed}
]

\node[box] (A) at (0,0) {Candidate\\opponents' maps};
\node[key] (B) at (3.2,0) {Forward closure\\state and filters};
\node[box] (C) at (6.4,0) {Deviation\\fixed maps};
\node[key] (D) at (9.6,0) {Best response\\noise-state linear};
\node[key, text width=2.8cm] (E) at (4.8,-2.05) {Fixed point\\Nash equilibrium};

\draw[arr] (A) -- (B);
\draw[arr] (B) -- (C);
\draw[arr] (C) -- (D);
\draw[arr] (D.north) -- ++(0,0.45) -| (A.north);
\draw[soft] (B.south) -- (E.north west);
\draw[soft] (D.south) -- (E.north east);

\end{tikzpicture}
\caption{Noise-state fixed-point loop. Opponents' strategy maps determine state and filter response functions; deviations propagate through these fixed maps; the FOC gives a noise-state linear best response.}
\label{fig:nsl_fixed_point_loop}
\end{figure}

\paragraph{Filtering closure.}
Recall the noise-state $\widehat W_t^i(\cdot)$~\eqref{eq:noise_state_def}. We record the deterministic-kernel representation of estimated-noise increments on the causal triangle $\Delta_T=\{(t,s):0\le s\le t\le T\}$, extended by zero off $\Delta_T$. Write $d_u$ for increments in the path index and $d_t$ for updates in time. The state has form $X_t=\bar X_t+\int_0^t X_t(s)\,dW_s$, with $X_t(s)$ the impulse-response kernel. Appendix~\ref{app:filtering} gives the general linear-filtering result; here $C(t,s)=\Gamma^i(t)X_t(s)$.

\begin{thm}[Linear Filtering Closure]
\label{thm:noise_state_filter_compact}
Fix a player $i$.
Assume the state admits the primitive-noise linear form
\[
X_t=\bar X_t+\int_0^t X_t(s)\,dW_s,
\]
and player $i$ observes
\[
dY_t^i=\Gamma^i(t)\,X_t\,dt+E^i\,dW_t,
\]
with deterministic $P^i(t)\succeq 0$ and block selector $E^i$.

Let $\Pi^i:=E^{i,\top}E^i$ and $\widehat X_t^i:=\E[X_t\mid\calF_t^i]$, and define the innovation
$dI_t^i:=dY_t^i-\Gamma^i(t)\,\widehat X_t^i\,dt$.

\medskip
\noindent
(i) The estimated-noise path increments admit the decomposition
\begin{equation}\label{eq:F_rep_du_What}
d_u\widehat W_t^i(u)
=
\underbrace{\Pi^i\,dW_u}_{\text{directly observed}}
+
\underbrace{\Big(\int_0^t F_t^i(u,s)\,dW_s\Big)\,du}_{\text{indirect inference}},
\qquad u<t,
\end{equation}
where the deterministic kernel $F^i$ is given explicitly by
\begin{equation}\label{eq:F_explicit}
F_t^i(u,s)
=
\widetilde X^i_s(u)^\top\Gamma^i(s)^\top E^i
\;+\;
E^{i,\top}\Gamma^i(u)\,\widetilde X^i_u(s)
\;+\;
\int_{\max(u,s)}^t \widetilde X^i_r(u)^\top P^i(r)\,\widetilde X^i_r(s)\,dr.
\end{equation}

\medskip
\noindent
(ii) The mixed $(t,u)$-update satisfies
\[
d_t\bigl(d_u\widehat W_t^i(u)\bigr)
=
\widetilde X^i_t(u)^\top\Gamma^i(t)^\top\,dI_t^i\,du.
\]

\medskip
\noindent
The unresolved kernel $\widetilde X^i$ is determined algebraically
from $F^i$ by~\eqref{eq:Xtilde_from_F}; the filtering kernel $F^i$
is the unique solution of the forward evolution
system~\eqref{eq:F_evolution}--\eqref{eq:F_boundary}
(Definition~\ref{def:filtering_kernel_system}).
\end{thm}

\paragraph{Best-response closure and information wedges.}
When opponents use noise-state linear strategies, a player's best response over the full admissible $L^2$ class remains noise-state linear. The maximum principle yields deterministic adjoint kernels and the information wedge $\mathcal{V}^i(t)$.

\begin{thm}[Multiplayer belief-adjoint kernels and information wedges]\label{thm:wedge-adjoints}
Fix a player $i$ and suppose each opponent $k\neq i$ uses a noise-state linear strategy with deterministic impulse-response maps, which pins down a deterministic forward environment
(state impulse response $X(\cdot,\cdot)$, filtering objects $(P^k,\widetilde X^k,F^k)$, and policy kernels $D^k$).
There exist linear processes $H^{X,i}_t$ and $H^{k,i}_t(u)$, $k\neq i$, such that the following system holds.
For each opponent $k\neq i$, define
\begin{equation}
\label{eq:wedge_component_main}
\mathcal V^{k,i}_t
:=
\Gamma^k(t)^\top E^kH^{k,i}_t(t)
+P^k(t)\int_0^t \widetilde X^k_t(z)\,H^{k,i}_t(z)\,dz,
\qquad
\mathcal V^i_t:=\sum_{k\neq i}\mathcal V^{k,i}_t .
\end{equation}
The physical costate and belief adjoints satisfy
\begin{align}
\frac{d}{dt}H^{X,i}_t
&=
\underbrace{\left[-\big(G^{XX,i}(t)X_t+G^{X,i}_t\big)
-A(t)^\top H^{X,i}_t\right]}_{\text{perfect-information costate block}}
-\underbrace{\mathcal V^i_t}_{\text{information wedge}}, \label{eq:HX}\\[0.3em]
\frac{d}{dt}H^{k,i}_t(u)
&=
-\big(D^k_t(u)\big)^\top H^{X,i}_t
+X_t(u)^\top\mathcal V^{k,i}_t,
\qquad k\neq i, \label{eq:Hk}
\end{align}
with terminal conditions
$H^{X,i}_T=G^{XX,i}(T)X_T+G^{X,i}_T$ and $H^{k,i}_T(\cdot)=0$.
All coefficients, $P^k(t)$, $\widetilde X^k_t(z)$, $D^k_t(u)$, and $X_t(u)$, are determined by opponents' frozen strategy maps and do not change when player~$i$ deviates.
\end{thm}

The bracketed block in~\eqref{eq:HX} is the perfect-information costate equation; the remaining term is the information wedge.  In~\eqref{eq:Hk}, $-D^k_t(u)^\top H^{X,i}_t$ transmits the physical costate into opponent~$k$'s belief channel.  The term $\Gamma^k(t)^\top E^kH^{k,i}_t(t)$ prices the newly born innovation coordinate, and the integral term prices accumulated drift inference.

For a fixed profile, the continuation value is a functional of the physical state and the players' noise states, with adjoints $H^{X,i}_t=\partial_XV^i_t$ and $H^{k,i}_t(u)=\delta V^i_t/\delta\widehat W^k_t(u)$.  The wedge is the chain-rule term generated when a physical perturbation changes opponents' future noise states.  Because all coefficients in~\eqref{eq:HX}--\eqref{eq:Hk} are deterministic, the mean and response-map coefficients decouple and are solved separately.  Both wedge components vanish when opponents' signals are exogenous to individual controls.

\begin{cor}[No observational externality]\label{cor:exogenous_signal}
Suppose a unilateral deviation in $D^i$ leaves every opponent's observation law, and hence the fixed filter for each $\widehat W^k$, $k\neq i$, unchanged.  Then $\mathcal{V}^i_t\equiv0$, the belief-adjoint equations~\eqref{eq:Hk} decouple from the physical costate, and the optimality system reduces to the standard LQG game with fixed information structures.
\end{cor}

\begin{lem}[Strict convexity of the best-response problem]
\label{lem:convexity_global}
Fix opponents' strategy maps $D^{-i}$ in the deterministic impulse-response class, including their fixed filters from signal histories to noise states.  Then player $i$'s induced objective
$D^i\mapsto J^i(D^i;D^{-i})$ is a strictly convex functional on the admissible $L^2$ control space.
Consequently, the spike-variation stationarity condition
\eqref{eq:FOC_control_revised} is not only necessary but sufficient: any admissible control satisfying it
is the unique global best response.
\end{lem}

\begin{proof}[Proof sketch]
With opponents' maps fixed and linear, the closed-loop drift of $X$ is affine in $D^i$,
so $(X,D^{-i})$ depend affinely on $D^i$.  The cost \eqref{eq:objective} is quadratic with
$G^{DD,i}(t)\succ 0$, hence strictly convex in $D^i$.
\end{proof}

\begin{cor}[Noise-state linear best response (global closure)]\label{cor:nsl-br}
Under the conditions of Theorem~\ref{thm:wedge-adjoints}, fix an opponents' profile $D^{-i}$ of
noise-state linear strategies, including the associated fixed filters (Definition~\ref{def:noise_state_linear_control}).
If $D^i$ is a best response to $D^{-i}$ over the full admissible $L^2$ control class, then for a.e.\ $t\in[0,T]$,
\[
G^{DD,i}(t)D_t^i+\E[H^{X,i}_t\mid\calF_t^i]=0,
\qquad\text{i.e.,}\qquad
D_t^i=-(G^{DD,i}(t))^{-1}\E[H^{X,i}_t\mid\calF_t^i].
\]
Since $H^{X,i}_t$ is a linear process (in the extended sense allowing kernels with $u>t$, truncated by conditioning on $\calF_t^i$) and $\E[H^{X,i}_t\mid\calF_t^i]=\bar H^{X,i}_t+\int_0^t H^{X,i}_t(u)\,d\widehat W_t^i(u)$, every best response over the full admissible class is itself a noise-state linear control.
\end{cor}

\begin{proof}
Strict convexity (Lemma~\ref{lem:convexity_global}) makes~\eqref{eq:FOC_control_revised} necessary and sufficient. With opponents' maps and filters fixed, the innovation $dI_t^i:=dY_t^i-\Gamma^i(t)\widehat X_t^i dt$ is a standard $(\calF^i,\Prob)$-Brownian motion whose law does not depend on $D^i$ (Liptser--Shiryaev~\cite{LiptserShiryaev}, Theorem~7.12), and $\widehat W^i$ is a deterministic linear functional of $I^i_{[0,t]}$. Thus the FOC expresses $D^i_t$ as a deterministic affine functional of $I^i$, i.e. in noise-state linear form. The filters are determined by the candidate profile, so equilibrium remains a fixed point.
\end{proof}

\begin{rem}\label{rem:restricted_NE_is_NE}
Since every best response over the full admissible $L^2$ class is noise-state linear
(Corollary~\ref{cor:nsl-br}), any Nash equilibrium in noise-state linear controls is automatically a Nash equilibrium
of the unrestricted game.
\end{rem}

\begin{rem}[Separation and the information wedge]\label{rem:separation_failure}
Standard LQG separation holds the feedback gain fixed as signal precision changes. Here the FOC, $D_t^i=-(G^{DD,i}(t))^{-1}\E[H^{X,i}_t\mid\calF_t^i]$, depends on the projected physical costate. When actions affect opponents' signals, that costate contains the wedge, which depends on opponents' filters and belief adjoints. Changing information therefore generally changes the policy rule itself, not only the posterior variance. Special symmetries could align the wedge with the posterior-state kernel, but this is not what happens in the benchmark. If opponents' signals are exogenous to individual controls, the standard fixed-information representation is recovered.
\end{rem}

\section{Conclusion}
\label{sec:conclusion}

Finite-player dynamic games with dispersed private information are difficult because actions affect both the physical state and opponents' learning, making the usual state variables nonrecursive. This paper replaces belief hierarchies with the noise state: beliefs about primitive uncertainty, with higher-order beliefs generated by composing response maps. The broader program is to use the same primitive-shock state in richer conditional-Gaussian linear economies, especially sparse macroeconomic networks where aggregate mean-field modes coexist with local informational wedges that need not average out absent symmetry.

In the LQG benchmark, the noise-state representation becomes explicit. Beliefs, value gradients, and policy rules are deterministic impulse-response functions, and equilibrium in noise-state linear strategies is a fixed point in those functions. Best-response closure verifies that any fixed point is a Nash equilibrium against arbitrary admissible $L^2$ deviations. The associated first-order system contains an information wedge, the shadow price of changing opponents' posteriors.

The two-player benchmark shows that this wedge is economically important. Pooling gains are mostly strategic rather than estimation-based. Optimal precision allocation can generate information starvation: concentrating signals on the efficient player benefits both sides, while informing the inefficient player increases bilateral waste. Signal precision also changes equilibrium policy rules themselves, so the usual separation logic fails.

\appendix

\section{Filtering}
\label{app:filtering}

We prove Theorem~\ref{thm:noise_state_filter_compact}: a generic linear-filtering result is specialized to the physical observation model and then converted from innovations to primitive-noise coordinates.

\subsection{Dynamics of the noise-state}
\label{app:What_dynamics}

Fix a player $i$.  Recall the noise-state
$\widehat W_t^i(u):=\E[W_u\mid\calF_t^i]$, $0\le u\le t$,
and the observation channel
\[
dY_t^i = \Big(\bar C(t)+\int_0^t C(t,s)\,dW_s\Big)\,dt + E^i\,dW_t,
\qquad Y_0^i=0,
\]
where $C(t,s)$ is deterministic.
Since $\bar C$ is deterministic, we may assume without loss that $\bar C\equiv 0$
by replacing $Y_t^i$ with $Y_t^i-\int_0^t\bar C(r)\,dr$.

\begin{prop}[Dynamics of $\widehat W_t^i(u)$]
\label{prop:What_dynamics}
There exist deterministic matrix-valued functions $B^Y(u,t)$ and $B(u;t,s)$
such that, for each fixed $u$,
\begin{equation}
\label{eq:What_dynamics_integral}
\widehat W_t^i(u) - \widehat W_{t_0}^i(u)
=
\int_{t_0}^t\!\!\int_0^r B(u;r,s)\,d\widehat W_r^i(s)\,dr
\;+\;
\int_{t_0}^t B^Y(u,r)\,dY_r^i,
\qquad t\ge t_0\ge u,
\end{equation}
or equivalently in differential form,
\begin{equation}
\label{eq:What_dynamics_ansatz}
d_t\widehat W_t^i(u)
=
\Big(\int_0^t B(u;t,s)\,d\widehat W_t^i(s)\Big)\,dt
+
B^Y(u,t)\,dY_t^i,
\end{equation}
with
\begin{align}
B^Y(u,t)
&=
\int_0^t \partial_s\mathsf{C}_t^i(u,s)\,C(t,s)^\top\,ds
\;+\; E^{i,\top}\mathbf{1}_{\{u\ge t\}},
\label{eq:BY_formula}\\[4pt]
B(u;t,s)
&=
-B^Y(u,t)\,C(t,s),
\label{eq:B_formula}
\end{align}
where $\mathsf{C}_t^i(u,s):=\Cov(W_u,W_s\mid\calF_t^i)$ is the deterministic
conditional error covariance kernel.
\end{prop}

\begin{proof}
We use the characterizing identity for conditional expectations via
exponential test martingales.
Fix a deterministic vector function $\vec r(\cdot)$ and define
$\varphi_t^r$ by $\varphi_0^r=1$ and
$d\varphi_t^r = i\,\varphi_t^r\,(dY_t^i)^\top\vec r(t)$.
A necessary and sufficient condition characterizing $\widehat W_t^i(u)$ is
\begin{equation}
\label{eq:char_identity}
\E\!\big[\widehat W_t^i(u)\,\varphi_t^r\big]
= \E\!\big[W_u\,\varphi_t^r\big]
\qquad\text{for all }\vec r(\cdot).
\end{equation}
We compute the time derivative of each side independently, then match.

\medskip\noindent\textit{Left-hand side.}
Apply It\^o's product rule to $\widehat W_t^i(u)\,\varphi_t^r$, using the
ansatz~\eqref{eq:What_dynamics_ansatz} and
$d\varphi_t^r = i\,\varphi_t^r(dY_t^i)^\top\vec r(t)$:
\[
d\big(\widehat W_t^i(u)\,\varphi_t^r\big)
=
i\,\widehat W_t^i(u)\,\varphi_t^r\,(dY_t^i)^\top\vec r(t)
+\varphi_t^r\,d_t\widehat W_t^i(u)
+d\langle\widehat W^i(u),\varphi^r\rangle_t.
\]
The quadratic covariation contributes
$d\langle\widehat W^i(u),\varphi^r\rangle_t
=i\,\varphi_t^r\,B^Y(u,t)\,\vec r(t)\,dt$.
Taking expectations and collecting the $dt$ terms gives
\begin{align}
\frac{d}{dt}\E\!\big[\widehat W_t^i(u)\,\varphi_t^r\big]
&=
i\,\E\!\bigg[\varphi_t^r
\Big(\widehat W_t^i(u)\!\int_0^t C(t,s)\,d\widehat W_t^i(s)^\top
+B^Y(u,t)\Big)\bigg]\vec r(t)
\notag\\
&\quad+
\E\!\bigg[\varphi_t^r\int_0^t\!\big(B(u;t,s)+B^Y(u,t)\,C(t,s)\big)
\,d\widehat W_t^i(s)\bigg].
\label{eq:LHS_derivative}
\end{align}

\medskip\noindent\textit{Right-hand side.}
Expand $\E[W_u\,\varphi_t^r]$ using the SDE for $\varphi^r$ and apply the
tower property in the form
$\E[\varphi_s^r\,L]=\E[\varphi_s^r\,\E[L\mid\calF_s^i]]$.
Define the conditional second-moment kernel
\[
\mathcal{M}_s^i(u,z)
:=\E[W_u W_z^\top\mid\calF_s^i]
=\mathsf{C}_s^i(u,z)+\widehat W_s^i(u)\,\widehat W_s^i(z)^\top.
\]
Then
\[
\E\!\Big[W_u\!\int_0^s C(s,z)\,dW_z^\top\;\Big|\;\calF_s^i\Big]
=
\int_0^s\partial_z \mathcal{M}_s^i(u,z)\,C(s,z)^\top\,dz.
\]
Splitting $\mathcal{M}_s^i=\mathsf{C}_s^i+\widehat W_s^i\,\widehat W_s^{i,\top}$ separates
this into a covariance piece and a product-of-means piece:
\begin{equation}
\label{eq:M_split}
\int_0^s\partial_z \mathcal{M}_s^i(u,z)\,C(s,z)^\top\,dz
=
\int_0^s\partial_z\mathsf{C}_s^i(u,z)\,C(s,z)^\top\,dz
\;+\;
\widehat W_s^i(u)\!\int_0^s C(s,z)\,d\widehat W_s^i(z)^\top.
\end{equation}
Differentiating $\E[W_u\,\varphi_t^r]$ and collecting $dt$ terms then gives
\begin{align}
\frac{d}{dt}\E\!\big[W_u\,\varphi_t^r\big]
&=
i\,\E\!\bigg[\varphi_t^r
\Big(\int_0^t\partial_s\mathsf{C}_t^i(u,s)\,C(t,s)^\top\,ds
\;+\;
\widehat W_t^i(u)\!\int_0^t C(t,s)\,d\widehat W_t^i(s)^\top
\;+\; E^{i,\top}\mathbf{1}_{\{u\ge t\}}\Big)\bigg]\vec r(t).
\label{eq:RHS_derivative}
\end{align}

\medskip\noindent\textit{Coefficient matching.}
Equating~\eqref{eq:LHS_derivative} and~\eqref{eq:RHS_derivative} for all
$\vec r(\cdot)$ and all $t$ yields two conditions.
Matching the terms multiplying $\vec r(t)$ (after canceling the common
$\widehat W_t^i(u)\int C\,d\widehat W^{i,\top}$ terms that appear on both
sides via~\eqref{eq:M_split}) gives~\eqref{eq:BY_formula}.
Requiring the residual drift in~\eqref{eq:LHS_derivative} to vanish gives
\[
\int_0^t\big(B(u;t,s)+B^Y(u,t)\,C(t,s)\big)\,d\widehat W_t^i(s)=0,
\]
hence $B(u;t,s)=-B^Y(u,t)\,C(t,s)$, which is~\eqref{eq:B_formula}.
\end{proof}

\subsection{Specialization to the physical observation model}
\label{sec:specialization}

We now specialize to the observation equation of Section~\ref{sec:model},
\[
dY_t^i = \Gamma^i(t)\,X_t\,dt + E^i\,dW_t,
\]
which corresponds to setting $C(t,s)=\Gamma^i(t)\,X_t(s)$ in the
general framework of Proposition~\ref{prop:What_dynamics}.

\paragraph{Innovation process.}
Define the innovation
\[
dI_t^i := dY_t^i - \Gamma^i(t)\,\widehat X_t^i\,dt,
\]
where $\widehat X_t^i:=\E[X_t\mid\calF_t^i]$.
By the innovations theorem for Gaussian linear filtering
(Liptser--Shiryaev~\cite{LiptserShiryaev}, Theorem~7.12),
$I^i$ is a standard $(\calF^i,\Prob)$-Brownian motion.

\paragraph{Innovation form of the dynamics.}
Substituting $C(t,s)=\Gamma^i(t)\,X_t(s)$ into
Proposition~\ref{prop:What_dynamics} and using
$dI_t^i = dY_t^i - \Gamma^i(t)\,\widehat X_t^i\,dt$,
the dynamics~\eqref{eq:What_dynamics_ansatz} become
\begin{equation}
\label{eq:What_innov_form}
d_t\widehat W_t^i(u)
=
\Big(\int_0^t\partial_s\mathsf{C}_t^i(u,s)\,X_t(s)^\top\,ds\Big)
\Gamma^i(t)\,dI_t^i\,
+\; E^{i,\top}\mathbf{1}_{\{u\ge t\}}\,dI_t^i.
\end{equation}

\paragraph{Identifying $\widetilde X^i$ as the filtering gain.}
For $u<t$, the indicator term in~\eqref{eq:What_innov_form} is locally
constant in $u$, so the mixed $(t,u)$-update takes the form
\begin{equation}
\label{eq:duWhat_update_by_innov}
d_t\!\bigl(d_u\widehat W_t^i(u)\bigr)
= \widetilde X^i_t(u)^\top \Gamma^i(t)^\top\,dI_t^i\,du,
\qquad 0 \le u < t \le T.
\end{equation}
The gain in~\eqref{eq:duWhat_update_by_innov} is the conditional
cross-covariance density $\Cov(dW_u,\,\Gamma^i(t)(X_t{-}\widehat X_t^i)\,dt
\mid\calF_t^i)/dt$. Since
$X_t-\widehat X_t^i = \int_0^t \widetilde X^i_t(s)\,dW_s$, It\^o isometry gives
$\widetilde X^i_t(u)^\top\Gamma^i(t)^\top$.

\paragraph{Algebraic identity for $\widetilde X^i$ in terms of $F^i$.}
Recall the induced estimated-state kernel
$\widehat X^i_t(s) := X_t(s)\Pi^i + \int_0^t X_t(z)\,F_t^i(z,s)\,dz$.
Since $\widetilde X^i_t(u) := X_t(u) - \widehat X^i_t(u)$,
\begin{equation}
\label{eq:Xtilde_from_F}
\widetilde X^i_t(u)
=
X_t(u)(I-\Pi^i)
-\int_0^t X_t(z)\,F_t^i(z,u)\,dz.
\end{equation}
Once $F^i$ is determined, $\widetilde X^i$ follows algebraically from~\eqref{eq:Xtilde_from_F}.

\subsection{The filtering kernel \texorpdfstring{$F^i$}{Fi}}
\label{sec:def_F_kernel}

Define $F^i$ by its evolution equation, with $\widetilde X^i$ given by~\eqref{eq:Xtilde_from_F}.

\begin{defn}[Filtering kernel system]
\label{def:filtering_kernel_system}
Fix player $i$.  A deterministic causal kernel $F^i$ on $\Delta_T$
solves the \emph{filtering kernel system} if, with $\widetilde X^i$
defined by~\eqref{eq:Xtilde_from_F}, it satisfies $F_0^i=0$ and
\begin{align}
\partial_t F_t^i(u,s)
&=
\widetilde X^i_t(u)^\top P^i(t)\,\widetilde X^i_t(s),
\qquad u,s < t,
\label{eq:F_evolution}\\[4pt]
F_t^i(u,t)
&=
\widetilde X^i_t(u)^\top\Gamma^i(t)^\top E^i,
\qquad
F_t^i(t,u) = E^{i,\top}\Gamma^i(t)\,\widetilde X^i_t(u).
\label{eq:F_boundary}
\end{align}
\end{defn}

Since~\eqref{eq:Xtilde_from_F} expresses $\widetilde X^i$ as a
deterministic linear functional of $F^i$,
the system~\eqref{eq:F_evolution}--\eqref{eq:F_boundary} is a
closed forward ODE for $F^i$ (quadratic through the substitution).

We now convert the innovations-based characterization of
Subsections~\ref{app:What_dynamics}--\ref{sec:specialization}
into primitive-noise coordinates.
The goal is to construct the deterministic kernel $F^i$ such that
for every $t\in[0,T]$ and $u<t$,
\[
d_u\widehat W_t^i(u)
=
\Pi^i\,dW_u
+
\Bigg(\int_0^t F_t^i(u,s)\,dW_s\Bigg)\,du.
\]

\subsubsection{Derivation}

Recall the innovation form of the estimated-noise dynamics
(equation~\eqref{eq:What_innov_form}):
for $u\le t$,
\[
d_u\widehat W_t^i(u)
=
E^{i,\top}\,dI_u^i
+
\Bigg(\int_u^t \widetilde X^i_s(u)^\top\Gamma^i(s)\,dI_s^i\Bigg)\,du,
\]
where the first term is the direct observation at time~$u$ and the
second is the accumulated drift-based revision over $(u,t]$.
To pass from innovations to primitive noise, we expand each $dI_s^i$
in terms of $dW$.

\paragraph{Expanding the innovation.}
The physical-measure decomposition of the innovation is
\[
dI_t^i
=
E^i\,dW_t
\;+\;
\Gamma^i(t)\,(X_t-\widehat X_t^i)\,dt.
\]
The first piece is the martingale component;
the second is absolutely continuous.
Substituting the linear forms and using
$\widetilde X^i_t(r):=X_t(r)-\widehat X^i_t(r)$
(the definition of~$\widetilde X^i$):
\begin{equation}\label{eq:innov_expansion}
dI_t^i
=
E^i\,dW_t
+
\Gamma^i(t)\int_0^t \widetilde X^i_t(r)\,dW_r\,dt.
\end{equation}

\paragraph{Expanding the direct term.}
Substituting~\eqref{eq:innov_expansion} at $t=u$ into
$E^{i,\top}\,dI_u^i$:
\begin{align}
E^{i,\top}\,dI_u^i
&=
\underbrace{E^{i,\top}E^i\,dW_u}_{=\;\Pi^i\,dW_u}
\;+\;
E^{i,\top}\Gamma^i(u)\int_0^u \widetilde X^i_u(r)\,dW_r\,du.
\label{eq:direct_term_expansion}
\end{align}
The second term is absolutely continuous in $u$ with primitive-noise density
$E^{i,\top}\Gamma^i(u)\,\widetilde X^i_u(s)$ for $s\le u$.

\paragraph{Expanding the integral term.}
Substituting~\eqref{eq:innov_expansion} into
$\int_u^t \widetilde X^i_s(u)^\top\Gamma^i(s)\,dI_s^i$:
\begin{align*}
\int_u^t \widetilde X^i_s(u)^\top\Gamma^i(s)^\top\,dI_s^i
&=
\underbrace{\int_u^t \widetilde X^i_s(u)^\top\Gamma^i(s)^\top\,E^i\,dW_s}_{\text{(a) martingale}}
\\
&\quad+\;
\underbrace{\int_u^t \widetilde X^i_s(u)^\top P^i(s)
\int_0^s \widetilde X^i_s(r)\,dW_r\,ds}_{\text{(b) abs.\ continuous}}.
\end{align*}
Term~(a) contributes $\widetilde X^i_s(u)^\top\Gamma^i(s)^\top E^i$ for
$s\in(u,t]$. Applying stochastic Fubini to term~(b):
\[
\text{(b)}
=
\int_0^t\Bigg(\int_{\max(u,r)}^t
\widetilde X^i_s(u)^\top P^i(s)\,\widetilde X^i_s(r)\,ds\Bigg)\,dW_r.
\]

\paragraph{Assembling the density.}
Collecting the coefficient of $dW_s\,du$ from all three
contributions and noting that causality ($\widetilde X^i_s(u)=0$ for $u>s$)
automatically zeroes the irrelevant indicators:

\begin{equation}\label{eq:F_explicit_app}
F_t^i(u,s)
=
\widetilde X^i_s(u)^\top\Gamma^i(s)^\top E^i
\;+\;
E^{i,\top}\Gamma^i(u)\,\widetilde X^i_u(s)
\;+\;
\int_{\max(u,s)}^t \widetilde X^i_r(u)^\top P^i(r)\,\widetilde X^i_r(s)\,dr.
\end{equation}

The three terms are direct observation, instantaneous drift revelation, and accumulated indirect inference.

\medskip
\noindent Existence of the kernel follows from joint Gaussian projection of $(W,Y^i)$, and uniqueness follows from It\^o isometry: if two deterministic causal kernels give the same density representation, their difference has zero $L^2$ norm on the causal triangle.

\section{Control First Order Conditions}
\label{sec:control_foc}

This appendix derives the stationarity condition~\eqref{eq:FOC_control_revised} and the closed backward system for the adjoint kernels. With opponents' strategy maps and filters fixed, Corollary~\ref{cor:nsl-br} makes the best-response problem an optimization over deterministic impulse-response maps weighting the exogenous innovation $I^i$.%
\footnote{Opponents' deterministic kernels make $I^i$ a standard Brownian motion independent of $D^i$ \cite[Theorem~7.12]{LiptserShiryaev}; Lemma~\ref{lem:convexity_global} gives strict convexity over the full $L^2$ class, so the FOC is necessary and sufficient.}

\subsection{The best-response problem in innovations coordinates}
\label{sec:br_innovations}

Fix a baseline noise-state linear profile and deviating player $i$. By Corollary~\ref{cor:nsl-br}, $\widehat W_t^i(\cdot)$ is a deterministic linear functional of the innovation $I^i_{[0,t]}$, whose law is invariant to $D^i$. Thus $D_t^i=\bar D^i_t+\int_0^t D^i_t(u)\,d_u\widehat W_t^i(u)$ is optimized over deterministic maps $(\bar D^i,D^i)$. Linearity of the state dynamics and $G^{DD,i}(t)\succ0$ make the induced objective strictly convex.

\subsection{Spike variation and the first-variation system}
\label{sec:spike_variation}

Fix $t\in[0,T)$ and $v\in L^2(\Omega,\calF_t^i;\R^d)$.
The spike perturbation
$D_s^{i,\rho,\e}:=D_s^i+\rho\,v\,\1_{[t,t+\e]}(s)$
induces the normalized first variation
$\delta X_s:=\lim_{\e\downarrow 0}\frac{1}{\e}\frac{\partial}{\partial\rho}X_s^{\rho,\e}\big|_{\rho=0}$
with $\delta X_t=v$.

Opponents apply fixed linear maps, so $v\mapsto(\delta X_s,\{\delta w_s^k(\cdot)\}_{k\neq i})$ is linear with deterministic coefficients. The deviator's own innovation is unaffected ($\delta I_s^i=0$), since $\delta X_s\in\calF_t^i$ implies $\delta\widehat X_s^i=\delta X_s$.

\paragraph{Opponents' filter response.}
For each $k\neq i$,
$\delta(d_u\widehat W_s^k(u))=\delta w_s^k(u)\,du$
with $\delta w_t^k(\cdot)\equiv 0$.
Existing coordinates satisfy
\begin{equation}
\label{eq:delta_w_dynamics_revised}
\partial_s\delta w_s^k(u)
=
\widetilde X^k_s(u)^\top P^k(s)
\Bigl(\delta X_s-\int_0^s X_s(r)\,\delta w_s^k(r)\,dr\Bigr),
\qquad 0\le u<s,
\end{equation}
and the coordinate born at $u=s$ satisfies
\begin{equation}
\label{eq:delta_w_birth_revised}
\delta w_s^k(s)
=
E^{k,\top}\Gamma^k(s)
\Bigl(\delta X_s-\int_0^s X_s(r)\,\delta w_s^k(r)\,dr\Bigr).
\end{equation}
The opponent control variation is
$\delta D_s^k=\int_0^s D^k_s(u)\,\delta w_s^k(u)\,du$,
giving the coupled state system
\begin{equation}
\label{eq:delta_X_dynamics_revised}
\frac{d}{ds}\delta X_s
=
A(s)\,\delta X_s
+\sum_{k\neq i}\int_0^s D^k_s(u)\,\delta w_s^k(u)\,du,
\qquad \delta X_t=v.
\end{equation}

\subsection{Transition kernels}
\label{sec:transition_kernels_revised}

The coupled system
\eqref{eq:delta_w_dynamics_revised}--\eqref{eq:delta_X_dynamics_revised},
with boundary condition \eqref{eq:delta_w_birth_revised},
has bounded deterministic coefficients and admits a unique
deterministic two-parameter evolution family.

\begin{defn}[Block transition kernels]
\label{def:block_transition}
For $s\ge t$, define $\Phi^{XX}(s,t)\in\R^{d\times d}$ and
$\Phi^{Xk}(s,t,u)$ ($k\neq i$, $u\in[0,t]$) by
\[
\delta X_s
=
\Phi^{XX}(s,t)\,v
+\sum_{k\neq i}\int_0^t \Phi^{Xk}(s,t,u)\,\eta^k(u)\,du,
\]
for initial data $\delta X_t=v$, $\delta w_t^k(\cdot)=\eta^k(\cdot)$.
In the spike deviation $\eta^k\equiv 0$, so
$\delta X_s=\Phi^{XX}(s,t)\,v$.
The notation $\Phi^{Xk}(s,t,t)$ denotes the right-boundary value for the coordinate born at time $t$.
\end{defn}

These blocks satisfy
\begin{align}
\frac{\partial}{\partial t}\Phi^{XX}(s,t)
&=
-\Phi^{XX}(s,t)A(t)
-\sum_{k\neq i}\Phi^{Xk}(s,t,t)E^{k,\top}\Gamma^k(t)
\notag\\[-1pt]
&\qquad
-\sum_{k\neq i}\int_0^t \Phi^{Xk}(s,t,z)\,\widetilde X^k_t(z)^\top P^k(t)\,dz,
\label{eq:PhiXX_revised}\\[3pt]
\frac{\partial}{\partial t}\Phi^{Xk}(s,t,u)
&=
-\Phi^{XX}(s,t)D^k_t(u)
+\Phi^{Xk}(s,t,t)E^{k,\top}\Gamma^k(t)X_t(u)
\notag\\[-1pt]
&\qquad
+\int_0^t \Phi^{Xk}(s,t,z)\,\widetilde X^k_t(z)^\top P^k(t)\,X_t(u)\,dz,
\label{eq:PhiXk_revised}
\end{align}
with $\Phi^{XX}(t,t)=I$ and $\Phi^{Xk}(t,t,u)=0$.
The terms involving $\Phi^{Xk}(s,t,t)E^{k,\top}\Gamma^k(t)$ come from the diagonal birth condition~\eqref{eq:delta_w_birth_revised}.

\subsection{First variation of costs and the adjoint kernels}
\label{sec:cost_variation_revised}

Fix player $i$ and suppress player superscripts on weight matrices.
Using $\delta X_s=\Phi^{XX}(s,t)v$ and the linear form of $X_s$,
the normalized cost variation is
\begin{equation}
\label{eq:first_variation_revised}
\delta J_t^i(v)
=
2\,\E\!\left[\left.
v^\top\Bigl(
G^{DD,i}(t)D_t^i
+\bar H^X_t
+\int_0^T H^X_t(u)\,d\widehat W_t^i(u)
\Bigr)
\ \right|\calF_t^i\right],
\end{equation}
where
\begin{align}
\bar H^X_t
&:=
\int_t^T\Phi^{XX}(s,t)^\top\bigl[G^{XX}(s)\bar X_s+\bar G^{X,i}_s\bigr]\,ds
+\Phi^{XX}(T,t)^\top\bigl[G^{XX,i}(T)\bar X_T+\bar G^{X,i}_T\bigr],
\label{eq:def_HbarX_revised}\\[3pt]
H^X_t(u)
&:=
\int_t^T\Phi^{XX}(s,t)^\top\bigl[G^{XX}(s)\,X_s(u)+G^{X,i}_s(u)\bigr]\,ds
+\Phi^{XX}(T,t)^\top\bigl[G^{XX,i}(T)\,X_T(u)+G^{X,i}_T(u)\bigr].
\label{eq:def_HX_revised}
\end{align}

\paragraph{Stationarity condition (globally necessary and sufficient).}
Setting $\delta J_t^i(v)=0$ for all $\calF_t^i$-measurable $v$ gives
\begin{equation}
\label{eq:FOC_control_revised}
G^{DD,i}(t)D_t^i
+\bar H^X_t
+\int_0^T H^X_t(u)\,d\widehat W_t^i(u)
=0
\qquad\text{a.s.,}
\end{equation}
with equilibrium coefficients
$\bar D^i_t=-G^{DD,i}(t)^{-1}\bar H^X_t$ and
$D^i_t(u)=-G^{DD,i}(t)^{-1}H^X_t(u)$.

\smallskip
Equivalently, $D_t^i=-(G^{DD,i}(t))^{-1}\E[\bar H^X_t
+\int_0^T H^X_t(u)\,dW_u\mid\calF_t^i]$:
the control is a conditional-expectation projection of the costate.  This is not a separation principle: the costate itself is an equilibrium object shaped by opponents' filters and belief adjoints.

\subsection{Closed backward system for the adjoints}
\label{sec:backward_system_revised}

Define the belief-adjoint coefficients
\begin{align}
\bar H^k_t(u)
&:=
\int_t^T\Phi^{Xk}(s,t,u)^\top\bigl[G^{XX}(s)\bar X_s+\bar G^{X,i}_s\bigr]ds
+\Phi^{Xk}(T,t,u)^\top\bigl[G^{XX,i}(T)\bar X_T+\bar G^{X,i}_T\bigr],
\label{eq:def_Hbark_revised}\\[3pt]
H^k_t(u,r)
&:=
\int_t^T\Phi^{Xk}(s,t,u)^\top\bigl[G^{XX}(s)X_s(r)+G^{X,i}_s(r)\bigr]ds
\notag\\
&\qquad
+\Phi^{Xk}(T,t,u)^\top\bigl[G^{XX,i}(T)X_T(r)+G^{X,i}_T(r)\bigr].
\label{eq:def_Hk_revised}
\end{align}
Equivalently, introduce the noise-linear adjoint processes
\begin{align*}
H^X_t&:=\bar H^X_t+\int_0^T H^X_t(r)\,dW_r,\\
H^k_t(u)&:=\bar H^k_t(u)+\int_0^T H^k_t(u,r)\,dW_r,
\qquad k\neq i.
\end{align*}
For each opponent $k\neq i$, set
\begin{equation}
\label{eq:wedge_process_component_revised}
\mathcal V^{k,i}_t
:=
\Gamma^k(t)^\top E^k H^k_t(t)
+P^k(t)\int_0^t \widetilde X^k_t(z)H^k_t(z)\,dz,
\qquad
\mathcal V^i_t:=\sum_{k\neq i}\mathcal V^{k,i}_t .
\end{equation}
The first term is the contribution of the coordinate born at $u=t$ in
\eqref{eq:delta_w_birth_revised}; the second is the accumulated drift-inference
term from \eqref{eq:delta_w_dynamics_revised}.  In coefficients,
\begin{align}
\bar{\mathcal V}^{k,i}_t
&=
\Gamma^k(t)^\top E^k\bar H^k_t(t)
+P^k(t)\int_0^t \widetilde X^k_t(z)\bar H^k_t(z)\,dz,
\label{eq:wedge_mean_component_revised}\\[2pt]
\mathcal V^{k,i}_t(r)
&=
\Gamma^k(t)^\top E^kH^k_t(t,r)
+P^k(t)\int_0^t \widetilde X^k_t(z)H^k_t(z,r)\,dz .
\label{eq:wedge_kernel_component_revised}
\end{align}

The adjoints satisfy the two backward equations
\begin{align}
\frac{d}{dt}H^X_t
&=
-\bigl[G^{XX,i}(t)X_t+G^{X,i}_t\bigr]
-A(t)^\top H^X_t
-\mathcal V^i_t,
\label{eq:HX_process_backward_revised}\\[3pt]
\frac{d}{dt}H^k_t(u)
&=
-D^k_t(u)^\top H^X_t
+X_t(u)^\top\mathcal V^{k,i}_t,
\qquad k\neq i.
\label{eq:Hk_process_backward_revised}
\end{align}
The terminal conditions are
\begin{equation}
\label{eq:process_terminal_revised}
H^X_T=G^{XX,i}(T)X_T+G^{X,i}_T,
\qquad
H^k_T(\cdot)=0,
\qquad k\neq i .
\end{equation}

\begingroup
\small
\singlespacing
\setlength{\bibsep}{0pt}
\bibliographystyle{plainnat}
\bibliography{sources_pruned}
\endgroup
\end{document}